\documentclass[a4paper]{article}
\usepackage[margin=30mm]{geometry}

\usepackage{amsmath, amssymb, amsthm}
\usepackage{mathtools}
\usepackage[dvipdfmx]{graphicx}
\usepackage{hyperref}
\usepackage{xurl}

\usepackage[capitalize]{cleveref}
\usepackage{bbm}
\usepackage{siunitx}
\usepackage{physics}
\usepackage{xr}
\usepackage{enumitem}
\usepackage{listings}

\newtheorem{theorem}{Theorem}[section]
\theoremstyle{definition}
\newtheorem{definition}[theorem]{Definition}

\newtheorem{lemma}[theorem]{Lemma}
\newtheorem{corollary}[theorem]{Corollary}

\theoremstyle{remark}
\newtheorem{remark}[theorem]{Remark}
\Crefname{equation}{}{}

\newcommand\A{\mathcal{A}}
\newcommand\B{\mathcal{B}}
\newcommand\Z{\mathbb{Z}}
\renewcommand\d{\mathrm{d}}
\newcommand\N{\mathbb{N}}
\newcommand\R{\mathbb{R}}
\newcommand\Q{\mathbb{Q}}
\newcommand\F{\mathcal{F}}

\newcommand\D{D[0, 1]}
\newcommand\1{\mathbbm{1}}
\newcommand\E{\mathbb{E}}
\newcommand\wt{\widetilde}
\renewcommand\d{\mathrm{d}}

\newcommand\bs\boldsymbol

\DeclareMathOperator\Normal{Normal}
\DeclareMathOperator\Trapezoidal{Trapezoidal}

\DeclareMathOperator\Mixture{Mixture}

\allowdisplaybreaks

\title{Extension of the one-sample Kolmogorov-Smirnov test}
\author{
  Atsushi Komaba%
  \thanks{
    Department of Radiology, Faculty of Medicine, University of Yamanashi.
    email: \href{mailto:fiveseven.lambda@gmail.com}{fiveseven.lambda@gmail.com}.
    orcid: \href{https://orcid.org/0000-0002-4935-3945}{0000-0002-4935-3945}.
  }
  \and Hisashi Johno%
  \thanks{
    Department of Radiology, Faculty of Medicine, University of Yamanashi.
    email: \href{mailto:johnoh@yamanashi.ac.jp}{johnoh@yamanashi.ac.jp}.
    orcid: \href{https://orcid.org/0000-0002-8058-1401}{0000-0002-8058-1401}.
  }
  \and Kazunori Nakamoto%
  \thanks{
    Center for Medical Education and Sciences, Faculty of Medicine, University of Yamanashi.
    email: \href{mailto:nakamoto@yamanashi.ac.jp}{nakamoto@yamanashi.ac.jp}.
    orcid: \href{https://orcid.org/0000-0002-5626-5804}{0000-0002-5626-5804}.
  }
}

\date{}

\begin{document}

\maketitle
\begin{abstract}
  We propose here a new goodness-of-fit test, named the one-sample OVL-$q$ test ($q=1, 2, \ldots$), which can be considered an extension of the one-sample Kolmogorov-Smirnov test (equivalent to the one-sample OVL-1 test).
  We have analyzed the asymptotic properties of the one-sample OVL-2 test statistic and enabled the calculation of asymptotic p-values for the test statistic.
  We further conducted numerical experiments and demonstrated that the one-sample OVL-2 test can sometimes exceed the detection power of conventional goodness-of-fit tests.
\end{abstract}

\noindent MSC2020 subject classifications: Primary 62G10; Secondary 62G20, 62-04.


\section{Introduction}
The Kolmogorov-Smirnov (KS) test is a nonparametric method used to determine whether a sample originates from a specific probability distribution (one-sample KS test) or to assess whether two samples come from the same distribution (two-sample KS test).
In our previous study, we devised an extended version of the two-sample KS test, named the (two-sample) OVL-$q$ test ($q=1,2,\ldots$), and demonstrated its utility particularly for the two-sample OVL-2 test \cite{{komaba22}}.

In this study, we extended the one-sample KS test using a similar approach to our previous work \cite{{komaba22}}, and named it the one-sample OVL-$q$ test ($q=1,2,\ldots$).
We analyze the asymptotic properties of the one-sample OVL-2 test statistic and calculate its asymptotic p-values.
Furthermore, we assess the detection power of the one-sample OVL-2 test relative to conventional goodness-of-fit tests by conducting numerical experiments.

In this paper, we describe the analytical framework in \Cref{sec:analytical}.
Experimental results are shown in \Cref{sec:experiments}.
Conclusion follows in \Cref{sec:conclusion}.
The proofs of \Cref{thm:D_q_complete_convergence,thm:D_q_distribution,thm:D2_FnF_limit} are given in \Cref{sec:proofs}.
The source code for the experiments in \Cref{sec:experiments} is provided in the \nameref{supplementary-material}.

\subsection*{General notation}\label{sec:notation}
We denote by $\Z$, $\N$, $\N_+$, $\Q$, and $\R$
the sets of integers, nonnegative integers, positive integers, rational numbers, and real numbers, respectively.
If $-\infty \le a \le b \le \infty$ and if there is no confusion,
we write $[a, b] \coloneq \{x : a \le x \le b\}$,
$[a, b) \coloneq \{x : a \le x < b\}$,
$(a, b] \coloneq \{x : a < x \le b\}$,
and $(a, b) \coloneq \{x : a < x < b\}$ as (extended) real intervals.
For $n \in \N_+$, let $\R^n$ be the Euclidean $n$-dimensional space and $\R_\le^n \coloneq \{(v_1, \ldots, v_n) \in \R^n : v_1 \le \cdots \le v_n\}$.
For a topological space $A$, we denote by $\B(A)$ the $\sigma$-algebra of Borel sets in $A$.
For a set $A$, $\# A$ denotes the cardinality of $A$.
For a real function $f$ on a set $A$ and $x,y\in A$, we write $f|_x^y = f(y) - f(x)$.
We denote by $\1_A$ the indicator function of a set $A$.
For a random variable $X$, $\E[X]$ denotes its expectation.

\section{Analytical framework}\label{sec:analytical}
Let $\F$ be the set of distribution functions on $\R$, 
where each $F\in\F$ is nondecreasing, is continuous from the right, 
and satisfies $F(-\infty)\coloneq\lim_{x\to -\infty}F(x) = 0$ and $F(\infty)\coloneq\lim_{x\to\infty}F(x) = 1$.
For $F, G\in\F$ and $q\in\N_{+}$, 
\begin{equation}\label{eq:D_q}
	D_q (F, G) \coloneq 1 - \inf_{\bs{v}\in\R_{\le}^q} r_{F,G}(\bs{v}),
\end{equation} 
where 
\begin{equation}\label{eq:r}
	r_{F,G}(\bs{v}) \coloneq \sum_{i = 0}^q
			\min\bigl\{ F|_{v_i}^{v_{i + 1}}, G|_{v_i}^{v_{i + 1}}\bigr\}
\end{equation}
for $\bs{v} = (v_1,\ldots, v_q)\in\R_{\le}^q$, $v_0 = -\infty$, and $v_{q+1} = \infty$.
Note that $r_{F,G}(\bs{v})\in [0, 1]$ for all $\bs{v}\in\R_\le^q$, so that $D_q (F, G)\in [0, 1]$.

\begin{theorem}\label{thm:KS=D_1}
{\normalfont (See \cite{molnar14} for reference.)}
The KS metric on $\F$ equals $D_1$, that is, 
\begin{equation*}
	\sup_{v\in\R}\lvert F(v)-G(v)\rvert = D_1 (F, G)
	\qquad (F, G\in \F).
\end{equation*}
\end{theorem}

\begin{theorem}\label{thm:(FD_q)_complete_metric}
For each $q\in\N_+$, $(\F, D_q)$ is a complete metric space.
\end{theorem}

\Cref{thm:KS=D_1} can be proved similarly as in the proof of \cite[Proposition 2.7]{komaba22}.
\Cref{thm:(FD_q)_complete_metric} will be proved in \Cref{sec:proof.(FD_q)metric}.
By these theorems, we can see that $D_q$ are extension of the KS metric.
Furthermore, by \Cref{thm:D1_Dq_qD1}, any $D_q$ and $D_1$ generate the same topology on $\F$.

\subsection{One-sample KS test and its extension}\label{sec:1sample_OVL}
As a null hypothesis $H_0$, we assume that $X_1, \ldots, X_n$ are independent and identically distributed (i.i.d.) random variables on a probability space $(\Omega, \A, P)$ with a given distribution function $F\in\F$.
Let $F_n$ be the corresponding empirical distribution function, i.e., 
\begin{equation}\label{eq:F_n_definition}
	F_n(x) \coloneq \frac{1}{n}\sum_{i=1}^n\1_{(-\infty,x]}(X_i)\qquad (x\in\R).
\end{equation}
Here we propose $D_q(F_n, F) \colon \Omega \to \R$ ($q\in\N_+$) as an extension of the one-sample KS test statistic, which equals $D_1 (F_n, F)$ by \Cref{thm:KS=D_1}.
The p-value (function) of the extended test 
is given by 
\begin{equation}
\label{eq:def.p_qmn}
	p_{q,n}(x) \coloneq P(x \le D_q(F_n, F))\qquad (x\in\R),
\end{equation}
and the upper limit of a $100(1-\alpha)$\% confidence interval ($0<\alpha<1$) of $D_q(F_n, F)$ is
\begin{equation}
\label{eq:def.l_qmn}
	u_{q,n}(\alpha) \coloneq \inf\{x\in\R: p_{q,n}(x) < \alpha\}.
\end{equation}
This can be regarded as the one-sample OVL-$q$ test since the (two-sample) OVL-$q$ test statistic $\rho_{q,m,n}$ in \cite[Definition 2.2]{komaba22} equals $1-D_q(F_{0,m}, F_{1,n})$ (see \cite[Definition 2.1]{komaba22}), whose p-value is equal to that of $D_q(F_{0,m}, F_{1,n})$ as described in \cite[Section 2.3]{komaba22}.

\begin{theorem}\label{thm:D_q_complete_convergence}
For each $q\in\N_+$, $D_q(F_n, F)$ converges completely to $0$ as $n\to\infty$, i.e., 
\begin{equation*}
	\sum_{n=1}^\infty P(D_q(F_n, F) > \epsilon) < \infty
\end{equation*}
for any $\epsilon>0$.
\end{theorem}

Note that complete convergence implies almost sure convergence, as described in \cite[Remark 4.4]{johno21}.

\begin{theorem}\label{thm:D_q_distribution}
	For each $q\in\N_+$, the distribution of $D_q(F_n, F)$ is the same for all continuous $F\in\F$.
\end{theorem}

\begin{theorem}\label{thm:D2_FnF_limit}
If $F\in\F$ is continuous on $\R$, 
\begin{equation}\label{eq:limD2}
\lim_{n\to\infty} P\biggl(D_2(F_n, F) \ge \frac{a}{\sqrt{n}}\biggr) = 2\sum_{i=1}^\infty (4i^2 a^2 -1) \exp(-2i^2 a^2)
\end{equation}
for any $a>0$.
\end{theorem}

Note that $P(D_2(F_n, F) \ge a/\sqrt{n})$ in \Cref{eq:limD2} is independent of any continuous $F\in\F$ by \Cref{thm:D_q_distribution}.
By this theorem, we have the asymptotic distribution function of $\sqrt{n} D_2(F_n, F)$:
\begin{equation}\label{eq:lim_sqn_D2}
\lim_{n \to \infty} P(\sqrt{n} D_2(F_n, F) \le a)
= 1 - 2 \sum_{i = 1}^\infty (4 i^2 a^2 - 1) \exp(-2 i^2 a^2).
\end{equation}

See \Cref{sec:proof.thm:D_q_complete_convergence,sec:proof.D_q_distribution,sec:proof_D2_FnF_limit} for the proofs of \Cref{thm:D_q_complete_convergence,thm:D_q_distribution,thm:D2_FnF_limit}, respectively.

\section{Numerical experiments}\label{sec:experiments}
We conducted a computer-based experiment to compare the statistical power of the one-sample OVL-2 test with that of conventional statistical tests, including the one-sample KS test.

Beforehand, we computed $p_{2,n}$, as defined in \cref{eq:def.p_qmn}, using the Monte Carlo method for $n = 2^3, 2^4, \ldots, 2^{12}$,
because exact p-values for the one-sample OVL-2 test could not be computed.
More specifically,
instead of $p_{2,n}$,
we used the empirical distribution function of $D_2(U_n, U)$
computed from 100,000 samples of size $n$ drawn from the standard uniform distribution,
whose distribution function is defined as
\begin{equation}\label{eq:uniform_distribution}
U(x) = \max\{0, \min\{x, 1\}\} \qquad (x\in\R).
\end{equation}
Here, $U_n$ represents $F_n$ in the case $F = U$.
The empirical distribution functions for $n = 2^3, 2^5, 2^7, 2^9$ are shown in \cref{fig:d2_under_h0} together with the theoretical asymptotic distribution function given in \cref{eq:lim_sqn_D2}.

\begin{figure}
  \centering
  \includegraphics[width=\textwidth]{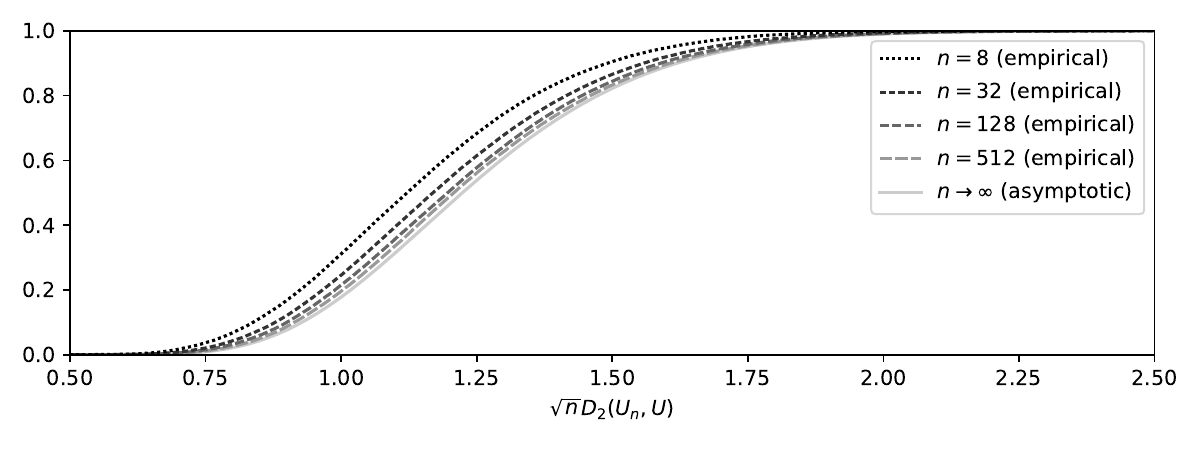}
  \caption{
    For each $n = 8, 32, 128, 512$,
    we generated 100,000 samples of size $n$ following the standard uniform distribution,
    and computed $D_2(U_n, U)$ for each sample.
	The empirical distribution function based on the 100,000 values of $\sqrt{n} D_2(U_n, U)$ is plotted in the graph,
	along with the right hand side of \cref{eq:lim_sqn_D2}, which represents the asymptotic distribution function where $n \to \infty$.
  }
  \label{fig:d2_under_h0}
\end{figure}

The probability density functions used in the experiments are defined as follows:
\begin{alignat*}{2}
\Normal(\mu, \sigma)(x) &= \frac{1}{\sqrt{2\pi} \sigma} \exp(-\frac{(x - \mu)^2}{2\sigma^2}) &&\quad\text{($\mu \in \R$, $\sigma > 0$, $x \in \R$),} \\
\Trapezoidal(x) &= \begin{cases}
  (x + 2) / 2 & \text{if $-2 \le x \le -\sqrt{2}$,} \\
  (2 - \sqrt{2}) / 2 & \text{if $-\sqrt{2} < x \le \sqrt{2}$,} \\
  (-x + 2) / 2 & \text{if $\sqrt{2} < x \le 2$,} \\
  0 & \text{if $x < -2$ or $2 < x$,}
\end{cases} &&\quad\text{($x \in \R$),} \\
\Mixture &= \frac{\Normal(-0.8, 0.6) + \Normal(0.8, 0.6)}{2},
\end{alignat*}
and they are illustrated in \cref{fig:dists_pdf}.

\begin{figure}
  \centering
  \includegraphics[width=\textwidth]{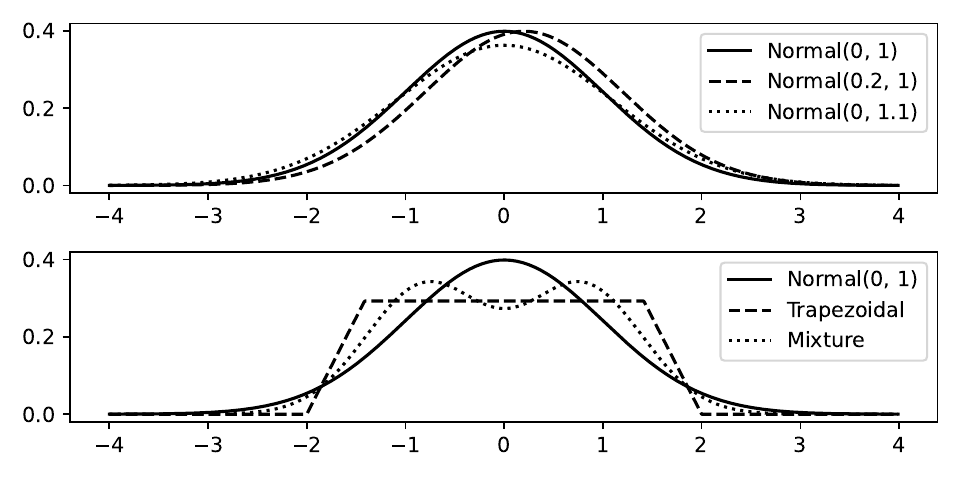}
  \caption{The probability density functions used in the experiments.}
  \label{fig:dists_pdf}
\end{figure}

First, we chose two different distributions as the sampling distribution and the reference distribution.
Then we repeated the following trial 100,000 times for each sample size $n = 2^3, 2^4, \ldots, 2^{12}$:
a random sample of size $n$ was drawn from the sampling distribution,
and tested under the null hypothesis that it were drawn from the reference distribution.
We performed the one-sample OVL-2 test, the one-sample KS test, and the Cram\'er-von Mises test for each sample,
and counted the number of times that the null hypothesis was rejected at 0.05 level of significance.
The rate of rejection out of 100,000 trials was assumed to represent the statistical power of the test.
The entire source code for the experiment, written in Python 3.11.8, is provided as the \nameref{supplementary-material}
on pages \pageref{supplementary-material}--\pageref{supplementary-material-end}.

The result is shown in \cref{fig:result1}.
When the sampling distribution was $\Normal(0.2, 1)$ and the reference distribution was $\Normal(0, 1)$,
the powers of the Cram\'er-von mises test and the one-sample KS test were respectively the first and second highest of the three, while the power of the one-sample OVL-2 test was lower.
When the sampling distribution was $\Normal(0, 1.1)$, $\Trapezoidal$, or $\Mixture$ and the reference distribution was $\Normal(0, 1)$,
the power of the one-sample OVL-2 test was the highest among the three, and the powers of the other two were almost equally lower.
When the sampling distribution was $\Trapezoidal$ and the reference distribution was $\Mixture$, or vice versa,
the power of the one-sample OVL-2 test was the highest, followed by that of the one-sample KS test, and that of the Cram\'er-von Mises test.

\section{Conclusion}\label{sec:conclusion}
In this study, we have developed the one-sample OVL-$q$ test ($q=1,2,\ldots$) as a new goodness-of-fit test, which can also be considered an extended version of the one-sample KS test (because the one-sample KS test is equivalent to the one-sample OVL-1 test).
We analyzed the asymptotic properties of the one-sample OVL-2 test statistic and enabled the calculation of its asymptotic p-values.

We conducted numerical experiments to compare the detection power of the one-sample OVL-2 test with conventional goodness-of-fit tests, including the one-sample KS test.
In several instances, the one-sample OVL-2 test demonstrated superior performance, suggesting its potential utility.

The limitations of this study are as follows:
\begin{itemize}
\item We have not presented a calculation method for the one-sample OVL-$q$ test when $q > 2$.
\item We calculated asymptotic p-values for the one-sample OVL-2 test statistic, but these are not applicable when the sample size is small.
\item We have not proposed a method for calculating p-values when the sample size is small, particularly exact methods.
\end{itemize}
To make the one-sample OVL-2 test practical, these issues need to be addressed in future research.

\begin{figure}
  \includegraphics[width=\textwidth]{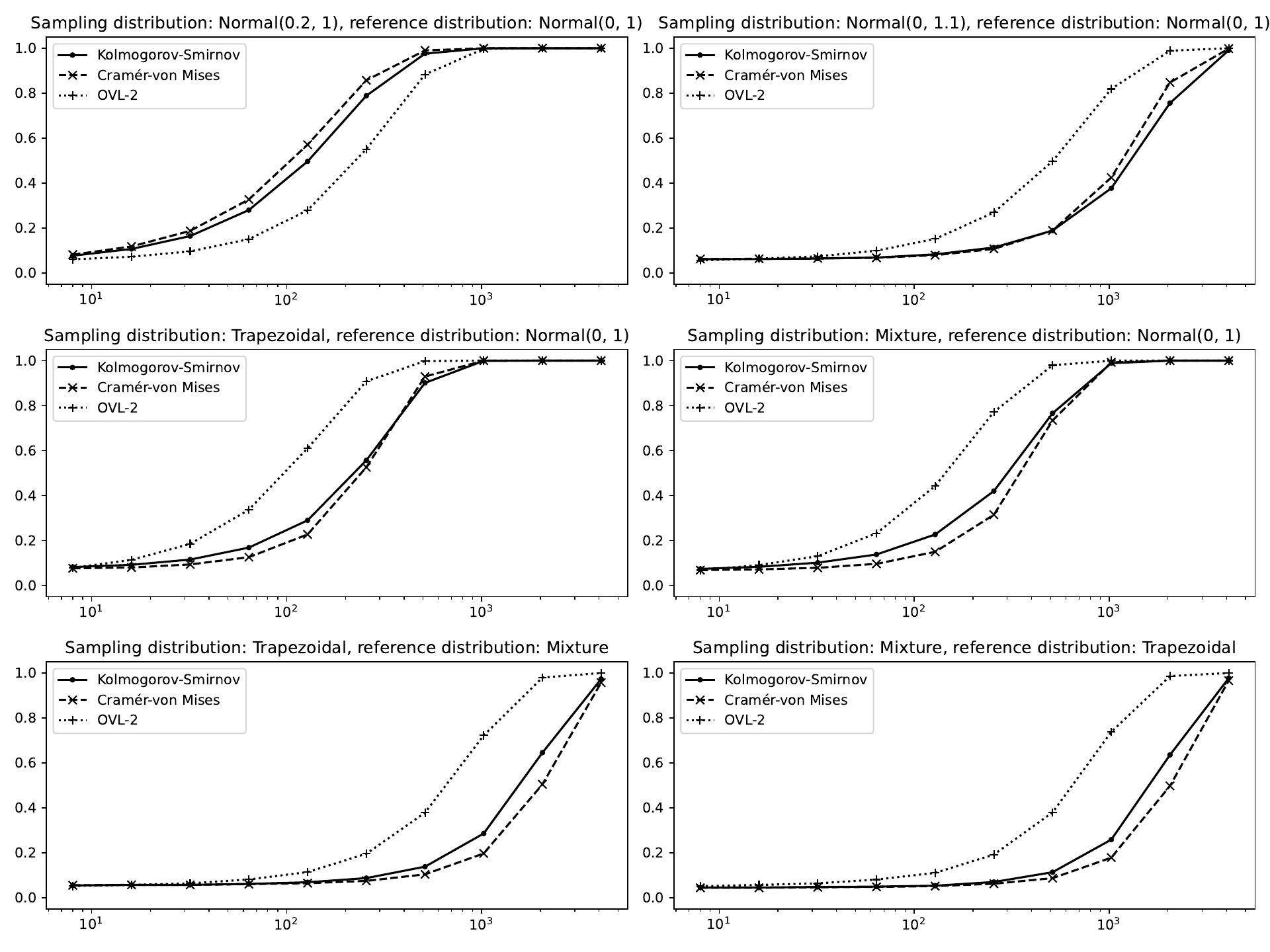}
  \caption{
    The statistical power of the one-sample OVL-2 test to detect the samples from the sampling distribution not following the reference distribution,
    compared to the statistical powers of the one-sample KS test and the Cram\'er-von Mises test.
    The horizontal and vertical axes of every graph represent the sample size and the statistical power, respectively.
  }
  \label{fig:result1}
\end{figure}

\section{Proofs}\label{sec:proofs}
\subsection{Proof of \texorpdfstring{\Cref{thm:(FD_q)_complete_metric}}{Theorem \ref{thm:(FD_q)_complete_metric}}}\label{sec:proof.(FD_q)metric}
Let us denote by $\F'$ the set of bounded right-continuous real functions on $\R$, and by $\lVert{\,\cdot\,}\rVert$ the supremum norm on $\F'$, i.e.,
\begin{equation}
	\lVert\xi\rVert \coloneq \sup_{x\in\R} \lvert\xi(x)\rvert < \infty \qquad (\xi\in\F').
\end{equation}

\begin{remark}
We can easily see that $\F'$ is a normed linear space with norm $\lVert{\,\cdot\,}\rVert$, and that $\F$ is a convex subset of $\F'$.
\end{remark}

\begin{theorem}\label{thm:Fd_complete_metric}
$\F'$ is a Banach space with norm $\lVert{\,\cdot\,}\rVert$.
\end{theorem}
\begin{proof}
Suppose $\{\xi_n\}$ is a Cauchy sequence in $\F'$.
For each $x\in\R$, $\lvert\xi_m(x)-\xi_n(x)\rvert\le \lVert\xi_m-\xi_n\rVert$ implies that $\{\xi_n(x)\}$ is a Cauchy sequence, so that there exists $\xi(x)\coloneq\lim_{n\to\infty}\xi_n(x)\in\R$ by the completeness of the real line.

For any $\epsilon>0$, there exists $N\in\N_+$ such that $m,n\ge N$ implies $\lVert\xi_m-\xi_n\rVert<\epsilon$, so that $\lvert\xi(x) - \xi_n(x)\rvert = \lim_{m\to\infty}\lvert\xi_m(x) - \xi_n(x)\rvert\le \epsilon$ for all $x\in\R$, i.e., $\lVert\xi-\xi_n\rVert\le\epsilon$, which also implies that $\lVert\xi\rVert\le\lVert\xi_n\rVert+\epsilon<\infty$.

For any $\epsilon>0$, there exists $n\in\N_+$ with $\lVert\xi-\xi_n\rVert\le\epsilon$ by the argument above.
For each $x\in\R$, there exists $\delta>0$ such that $\lvert\xi_n(x)-\xi_n(y)\rvert<\epsilon$ for all $y\in (x, x+\delta)$ since $\xi_n$ is right-continuous, so that
\begin{align*}
	\lvert\xi(x)-\xi(y)\rvert
	&= \lvert\xi(x)-\xi_n(x)+\xi_n(x)-\xi_n(y)+\xi_n(y)-\xi(y)\rvert \\
	&\le \lvert\xi(x)-\xi_n(x)\rvert + \lvert\xi_n(x)-\xi_n(y)\rvert + \lvert\xi_n(y)-\xi(y)\rvert \\
	&< 3\epsilon
\end{align*}
for all $y\in (x, x+\delta)$.
Hence $\xi$ is right-continuous.

Now we see that $\xi\in\F'$ and $\lim_{n\to\infty}\lVert\xi-\xi_n\rVert=0$, and the proof is complete.
\end{proof}

\begin{lemma}\label{lem:F_closed}
$\F$ is a closed subspace of $(\F', \lVert{\,\cdot\,}\rVert)$. 
\end{lemma}
\begin{proof}
Let $\{\xi_n\}$ be a convergent sequence in $\F\subset\F'$ and $\xi\coloneq\lim_{n\to\infty}\xi_n\in\F'$.

For each $x\in\R$, $\lvert\xi_n(x)-\xi(x)\rvert \le \lVert\xi_n-\xi\rVert\to 0$ ($n\to\infty$) implies that $\xi (x)=\lim_{n\to\infty}\xi_n(x)$.
Hence $\xi(x)\le\xi(y)$ for any $x<y$, since $\xi_n(x)\le\xi_n(y)$ for all $n$.
This means that $\xi$ is nondecreasing.

For any $\epsilon>0$,  there exists $n\in\N_+$ with $\lVert\xi_n-\xi\rVert<\epsilon$.
Since $\lim_{x\to -\infty}\xi_n (x) = 0$, there exists $M\in\R$ such that $\lvert\xi_n(x)\rvert<\epsilon$ for all $x<M$.
Hence $\lvert\xi(x)\rvert = \lvert\xi(x) - \xi_n(x) + \xi_n(x)\rvert\le \lvert\xi(x) - \xi_n(x)\rvert + \lvert\xi_n(x)\rvert < \lVert\xi - \xi_n\rVert + \epsilon < 2\epsilon$ for all $x<M$.
Therefore, $\lim_{x\to -\infty}\xi (x) = 0$.
The proof for $\lim_{x\to \infty}\xi (x) = 1$ is similar.

Taken together, we have shown that $\xi\in\F$.
\end{proof}

The following theorem follows immediately from \Cref{thm:KS=D_1,thm:Fd_complete_metric,lem:F_closed}.

\begin{theorem}\label{thm:F_D1_complete}
$(\F, D_1)$ is a complete metric space.
\end{theorem}

\begin{remark}
We can see that $(\F, D_1)$ is not separable.
For example, the collection of open subsets \[\{F\in\F: D_1(F, \1_{[a, \infty)})<1/2\}\qquad (a\in\R)\] is pairwise disjoint and uncountable. Here note that $D_1(\1_{[a, \infty)}, \1_{[b, \infty)})=1$ if $a\ne b$.
\end{remark}

\begin{lemma}\label{lem:min_inequality}
For any $a,b,c\in\R$, $\min\{a, b\} + \min\{b, c\} \le \min\{a, c\} + b$.
\end{lemma}
\begin{proof}
We can assume that $a\le c$ without loss of generality.
If $a\le b$, then $\min\{a, b\} + \min\{b, c\} = \min\{a, c\} + \min\{b, c\} \le \min\{a, c\} + b$.
If $a\ge b$, then $\min\{a, b\} + \min\{b, c\} = b + \min\{b, c\} \le b + \min\{a, c\}$.
\end{proof}

\begin{theorem}\label{thm:(FD_q)_metric}
For each $q\in\N_+$, $(\F, D_q)$ is a metric space.
\end{theorem}

\begin{proof}
We have to show that, for all $F, G, H\in\F$, 
\begin{enumerate}[label = (\alph*)]
\item $0\le D_q(F, G)<\infty$.\label{enum:D_q_metric_a}
\item $D_q(F, G) = 0$ if and only if $F = G$.\label{enum:D_q_metric_b}
\item $D_q(F, G) = D_q(G, F)$.\label{enum:D_q_metric_c}
\item $D_q(F, G) \le D_q(F, H) + D_q(H, G)$.\label{enum:D_q_metric_d}
\end{enumerate}

\ref{enum:D_q_metric_a} and \ref{enum:D_q_metric_c} follows from definition.

Let us start with \ref{enum:D_q_metric_b}.
If $F=G$, then $r_{F, G}(\bs{v})=1$ for any $\bs{v}\in\R_\le^q$, so that $D_q (F, G) = 0$, by definition.
If $D_q (F, G) = 0$, then $\inf_{\bs{v}\in\R_{\le}^q} r_{F,G}(\bs{v}) = 1$, so that $r_{F,G}(\bs{v}) = 1$ for all $\bs{v}\in\R_{\le}^q$, which implies $F=G$ by the following arguments.
If $F(x) < G(x)$ for some $x\in\R$, then for $\bs{v} = (x,\ldots,x)\in\R_\le^q$, we have 
\begin{align*}
	r_{F,G}(\bs{v}) &= \min\bigl\{F|_{-\infty}^x, G|_{-\infty}^x\bigr\} + \min\bigl\{F|_x^\infty, G|_x^\infty\bigr\}\\
	&= \min\{F(x), G(x) \} + \min\{1-F(x), 1-G(x)\}\\
	&= F(x) - G(x) + 1\\
	&< 1.
\end{align*}
Hence $F\ge G$ if $D_q (F, G) = 0$.
Similarly, $F\le G$ if $D_q (F, G) = 0$, proving \ref{enum:D_q_metric_b}.

As for \ref{enum:D_q_metric_d}, we have 
\begin{align*}
	&\inf_{\bs{v}\in\R_{\le}^q} r_{F,H}(\bs{v}) + \inf_{\bs{v}\in\R_{\le}^q} r_{H,G}(\bs{v})\\
	&= \inf_{\bs{v}\in\R_{\le}^q} \sum_{i = 0}^q\min\bigl\{ F|_{v_i}^{v_{i + 1}}, H|_{v_i}^{v_{i + 1}}\bigr\} + \inf_{\bs{v}\in\R_{\le}^q} \sum_{i = 0}^q\min\bigl\{ H|_{v_i}^{v_{i + 1}}, G|_{v_i}^{v_{i + 1}}\bigr\}\\
	&\le \inf_{\bs{v}\in\R_{\le}^q} \sum_{i = 0}^q
			\bigl(\min\bigl\{ F|_{v_i}^{v_{i + 1}}, H|_{v_i}^{v_{i + 1}}\bigr\} + \min\bigl\{ H|_{v_i}^{v_{i + 1}}, G|_{v_i}^{v_{i + 1}}\bigr\}\bigr)\\
	&\le \inf_{\bs{v}\in\R_{\le}^q} \sum_{i = 0}^q\bigl(\min\bigl\{ F|_{v_i}^{v_{i + 1}}, G|_{v_i}^{v_{i + 1}}\bigr\} +  H|_{v_i}^{v_{i + 1}}\bigr)\\
	&= \inf_{\bs{v}\in\R_{\le}^q} r_{F,G}(\bs{v}) + 1
\end{align*}
by \Cref{eq:r,lem:min_inequality}, so that $D_q(F, G) \le D_q(F, H) + D_q(H, G)$.
This completes the proof.
\end{proof}

\begin{lemma}\label{lem:Dq_le_Dqd}
For any $F, G\in\F$, $D_q(F, G)\le D_{q'}(F, G)$ if $q<q'$.
\end{lemma}
\begin{proof}
Since $\inf_{\bs{v}\in\R_{\le}^{q'}} r_{F,G}(\bs{v})\le \inf_{\bs{v}\in\R_{\le}^q} r_{F,G}(\bs{v})$ by definition, $D_q(F, G)\le D_{q'}(F, G)$ holds.
\end{proof}

\begin{lemma}\label{lem:Dq_qD1}
For each $q\in\N_+$, $D_q(F, G)\le q D_1(F, G)$ for all $F, G\in\F$.
\end{lemma}
\begin{proof}
It follows from definition that
\begin{align*}
	D_q (F, G) &= 1 - \inf_{\bs{v}\in\R_{\le}^q} \sum_{i = 0}^q
			\min\bigl\{ F|_{v_i}^{v_{i + 1}}, G|_{v_i}^{v_{i + 1}}\bigr\} \\
			&= 1 - \inf_{\bs{v}\in\R_{\le}^q} \sum_{i = 0}^q \frac{1}{2}\bigl(
			 F|_{v_i}^{v_{i + 1}} + G|_{v_i}^{v_{i + 1}} - \bigl\lvert F|_{v_i}^{v_{i + 1}} - G|_{v_i}^{v_{i + 1}}\bigr\rvert\bigr) \\
			 &= 1 - \inf_{\bs{v}\in\R_{\le}^q} \biggl(1 - \sum_{i = 0}^q \frac{1}{2}\bigl\lvert F|_{v_i}^{v_{i + 1}} - G|_{v_i}^{v_{i + 1}}\bigr\rvert\biggr) \\
			 &=  \sup_{\bs{v}\in\R_{\le}^q} \sum_{i = 0}^q \frac{1}{2} \bigl\lvert F|_{v_i}^{v_{i + 1}} - G|_{v_i}^{v_{i + 1}}\bigr\rvert \\
			 &\le \sup_{\bs{v}\in\R_{\le}^q} \sum_{i = 1}^q \lvert F(v_i) - G(v_i)\rvert \\
			 &\le q \sup_{x\in\R}\lvert F(x) - G(x)\rvert \\
			 &= q D_1(F, G),
\end{align*}
and the proof is complete.
\end{proof}

\begin{remark}\label{rem:another_Dq_def}
As shown in the proof above, we obtain the equation
\begin{equation*}
	D_q (F, G) = \frac{1}{2}  \sup_{\bs{v}\in\R_{\le}^q} \sum_{i = 0}^q \bigl\lvert F|_{v_i}^{v_{i + 1}} - G|_{v_i}^{v_{i + 1}}\bigr\rvert \qquad (F, G\in\F).
\end{equation*}
\end{remark}

The following theorem follows immediately from \Cref{lem:Dq_le_Dqd,lem:Dq_qD1}.

\begin{theorem}\label{thm:D1_Dq_qD1}
For each $q\in\N_+$, $D_1(F, G) \le D_q(F, G)\le q D_1(F, G)$ for all $F, G\in\F$.
\end{theorem}

Now \Cref{thm:(FD_q)_complete_metric} follows from \Cref{thm:F_D1_complete,thm:(FD_q)_metric,thm:D1_Dq_qD1}.

\subsection{Proof of \texorpdfstring{\Cref{thm:D_q_complete_convergence}}{Theorem \ref{thm:D_q_complete_convergence}}}\label{sec:proof.thm:D_q_complete_convergence}

\begin{theorem}\label{thm:Fn_complete_convergence}
{\normalfont (The Glivenko-Cantelli theorem. See the proof of \cite[Theorem A, Section 2.1.4]{serfling80}.)}
$\sup_{x\in\R}\lvert F_n(x)-F(x)\rvert$ converges completely to 0 as $n\to\infty$, i.e., 
\begin{equation*}
	\sum_{n=1}^\infty P\biggl(\sup_{x\in\R}\lvert F_n(x)-F(x)\rvert> \epsilon\biggr) < \infty
\end{equation*}
for any $\epsilon>0$.
\end{theorem}

It follows from \Cref{lem:Dq_qD1,thm:Fn_complete_convergence} that
\begin{align*}
	\sum_{n=1}^\infty P(D_q(F_n, F) > \epsilon)
	&\le \sum_{n=1}^\infty P\biggl(q \sup_{x\in\R}\lvert F_n(x) - F(x)\rvert  > \epsilon\biggr)\\
	&= \sum_{n=1}^\infty P\biggl(\sup_{x\in\R}\lvert F_n(x) - F(x)\rvert  > \epsilon/q\biggr)\\
	&<\infty
\end{align*}
for any $\epsilon>0$.
This proves \Cref{thm:D_q_complete_convergence}.

\subsection{Proof of \texorpdfstring{\Cref{thm:D_q_distribution}}{Theorem \ref{thm:D_q_distribution}}}\label{sec:proof.D_q_distribution}
Let us denote by $F^-$ the quantile function of $F\in\F$, i.e., 
\begin{equation}\label{eq:quantile_function}
	F^- (y) \coloneq \inf \{ x\in\R : F(x)\ge y \} \qquad (y\in [0, 1])
\end{equation}
with the convention that $\inf \emptyset \coloneq \infty$ and $\inf \R \coloneq -\infty$.
Let $U$ be the standard uniform distribution function defined in \cref{eq:uniform_distribution}.

\begin{theorem}\label{thm:quantile_function_properties}
{\normalfont (See \cite[Proposition 1.1]{Dudley_2014} or \cite[Propositions 1 and 2]{embrechts13} for reference.)}
\begin{enumerate}[label=(\alph*)]
\item For any $y\in (0, 1)$, $F^-(y)$ is a finite real number.
\item For any $y\in (0, 1)$, $F(F^-(y))\ge y$.
\item For any $x\in\R$ and $y\in (0, 1)$, $F(x)\ge y$ if and only if $x\ge F^-(y)$.
\item If the distribution function of $Z$ is $U$, then the distribution function of $F^-(Z)$ is $F$.
\end{enumerate}
\end{theorem}

Let $W_1, \ldots, W_n$ be i.i.d.\ random variables on a probability space $(\Omega', \A', P')$ with $U$, $X_i' \coloneq F^-(W_i)$ for $i=1,\ldots,n$, and 
\begin{align*}
	U_n(x) &\coloneq \frac{1}{n}\sum_{i=1}^n\1_{(-\infty,x]}(W_i)\qquad (x\in\R),\\
	F_n'(x) &\coloneq \frac{1}{n}\sum_{i=1}^n\1_{(-\infty,x]}(X_i')\qquad (x\in\R).
\end{align*}
As described in \Cref{sec:1sample_OVL}, $X_1, \ldots, X_n$ are i.i.d.\ random variables on $(\Omega, \A, P)$ with $F$ and 
\begin{equation*}
	F_n(x) \coloneq \frac{1}{n}\sum_{i=1}^n\1_{(-\infty,x]}(X_i)\qquad (x\in\R).
\end{equation*}
The following colloraries are immediate consequences of  \Cref{thm:quantile_function_properties}.

\begin{corollary}\label{cor:P(Xd)=P(X)}
The random variables $X_1',\ldots,X_n'$ are i.i.d. with the same distribution function $F$.
The probability measure on $\B(\R^n)$ induced by $(X_1',\ldots,X_n')$ and that by $(X_1,\ldots,X_n)$ are the same, i.e., 
\begin{equation*}
	P'((X_1',\ldots,X_n')\in A) = P((X_1,\ldots,X_n)\in A) \qquad (A\in\B(\R^n)).
\end{equation*}
\end{corollary}

\begin{corollary}\label{cor:Fn=Un_F}
It holds almost surely that $F_n' = U_n\circ F$.
\end{corollary}

Note that $F = U\circ F$ holds obviously.

\begin{theorem}\label{thm:Dq'=Dq}
For each $q\in\N_+$, $D_q (F_n', F) \le D_q (U_n, U)$ almost surely.
If $F$ is continuous on $\R$, $D_q (F_n', F) = D_q (U_n, U)$ almost surely.
\end{theorem}
\begin{proof}
For $\bs{v} = (v_1,\ldots, v_q)\in\R_{\le}^q$, it follows from \Cref{eq:r,cor:Fn=Un_F} that
\begin{align*}
	r_{F_n',F}(\bs{v}) 
	&= \sum_{i = 0}^q
			\min\bigl\{ F_n'|_{v_i}^{v_{i + 1}}, F|_{v_i}^{v_{i + 1}}\bigr\} \\
	&= \sum_{i = 0}^q
			\min\bigl\{ U_n\circ F|_{v_i}^{v_{i + 1}}, U\circ F|_{v_i}^{v_{i + 1}}\bigr\} \\
	&= \sum_{i = 0}^q
			\min\Bigl\{ U_n|_{F(v_i)}^{F(v_{i + 1})}, U|_{F(v_i)}^{F(v_{i + 1})}\Bigr\},
\end{align*}
where $F(v_0)=F(-\infty)=0$ and $F(v_{q+1})=F(\infty)=1$.
Since $U_n(0) = 0 = U_n(-\infty)$ and $U_n(1) = 1 = U_n(\infty)$ with probability 1, we have
\begin{equation*}
	r_{F_n', F}(\bs{v}) = r_{U_n, U}(F(\bs{v})),
	 \qquad F(\bs{v})\coloneq (F(v_1),\ldots,F(v_n))\in\R_{\le}^q
\end{equation*}
almost surely.
Hence
\begin{equation*}
	 \inf_{\bs{v}\in\R_{\le}^q} r_{F_n',F}(\bs{v})
	 = \inf_{\bs{v}\in\R_{\le}^q} r_{U_n, U}(F(\bs{v}))
	 \ge  \inf_{\bs{v}\in\R_{\le}^q} r_{U_n, U}(\bs{v})
\end{equation*}
and $D_q(F_n', F) \le D_q(U_n, U)$ almost surely.
If $F$ is continuous on $\R$, we have $F(\R)\supset (0, 1)$, so that 
\begin{equation*}
	\inf_{\bs{v}\in\R_{\le}^q} r_{U_n, U}(F(\bs{v}))
	=\inf_{\bs{v}\in\R_{\le}^q} r_{U_n, U}(\bs{v})
\end{equation*}
and $D_q(F_n', F) = D_q(U_n, U)$ almost surely.
\end{proof}

Let us define, for each $(t_1,\ldots,t_n)\in\R^n$, 
\begin{equation*}
	\Phi_{(t_1,\ldots,t_n)}(x) \coloneq \frac{1}{n}\sum_{i=1}^n\1_{(-\infty,x]}(t_i)\qquad (x\in\R).
\end{equation*}
We also define, for each $q\in\N_+$ and $\xi\in\F$, 
\begin{equation*}
	\wt{D}_{q,\xi}(\bs{t}) \coloneq D_q(\Phi_{\bs{t}}, \xi) \qquad (\bs{t}\in\R^n).
\end{equation*}

\begin{remark}\label{rem:Dq_composite}
We see that 
$D_q(U_n,U) = \wt{D}_{q,U}\circ (W_1,\ldots,W_n)$ and
$D_q(F_n',F) = \wt{D}_{q,F}\circ (X_1',\ldots,X_n')$ on $(\Omega',\A',P')$,
and $D_q(F_n,F) = \wt{D}_{q,F}\circ (X_1,\ldots,X_n)$ on $(\Omega,\A,P)$.
\end{remark}

\begin{theorem}\label{thm:Dq_xi_measurable}
For each $q\in\N_+$ and $\xi\in\F$, $\wt{D}_{q,\xi}$ is a Borel measurable function on $\R^n$.
\end{theorem}
\begin{proof}
Let us put $\Q_\le^q \coloneq \R_\le^q\cap\Q^q$.
It is obvious that
\begin{equation}\label{eq:r_R_le_Q}
	\inf_{\bs{v}\in\R_\le^q}r_{\Phi_{\bs{t}},\xi}(\bs{v}) \le \inf_{\bs{v}\in\Q_\le^q}r_{\Phi_{\bs{t}},\xi}(\bs{v}) \qquad (\bs{t}\in\R^n).
\end{equation}
For any $\epsilon>0$, there exists $\bs{x}\coloneq (x_1,\ldots,x_q)\in\R_\le^q$ such that $r_{\Phi_{\bs{t}},\xi}(\bs{x})<\inf_{\bs{v}\in\R_\le^q}r_{\Phi_{\bs{t}},\xi}(\bs{v}) + \epsilon$.
Let $\{\bs{a}_i\coloneq (a_{i,1},\ldots,a_{i,q})\}$ be a sequence in $\Q_\le^q$ such that for each $j\in\{1,\ldots,q\}$, $a_{i,j}$ converges to $x_j$ from the right as $i\to\infty$.
Since $\Phi_{\bs{t}}$ and $\xi$ are right-continuous on $\R$, $\lim_{i\to\infty}r_{\Phi_{\bs{t}},\xi}(\bs{a}_i)=r_{\Phi_{\bs{t}},\xi}(\bs{x})$.
Hence $r_{\Phi_{\bs{t}},\xi}(\bs{b})<\inf_{\bs{v}\in\R_\le^q}r_{\Phi_{\bs{t}},\xi}(\bs{v}) + \epsilon$ for some $\bs{b}\in\Q_\le^q$.
With \Cref{eq:r_R_le_Q}, we have
\begin{equation}\label{eq:r_R_=_Q}
	\inf_{\bs{v}\in\R_\le^q}r_{\Phi_{\bs{t}},\xi}(\bs{v}) = \inf_{\bs{v}\in\Q_\le^q}r_{\Phi_{\bs{t}},\xi}(\bs{v}) \qquad (\bs{t}\in\R^n),
\end{equation}
so that
\begin{equation}\label{eq:wtD_countable}
	\wt{D}_{q,\xi}(\bs{t}) = 1-\inf_{\bs{v}\in\Q_\le^q}r_{\Phi_{\bs{t}},\xi}(\bs{v}) \qquad (\bs{t}\in\R^n).
\end{equation}
by definition.

It is immediate from definition that $\Phi_\bullet(x)\colon\R^n\to\R$ and $r_{\Phi_\bullet,\xi}(\bs{v})\colon\R^n\to\R$ are Borel measurable for each $x\in\R$ and $\bs{v}\in\R_\le^q$, respectively.
Hence $\wt{D}_{q,\xi}\colon\R^n\to\R$ can be described by the countable infimum of Borel measurable functions by \Cref{eq:wtD_countable}.
This implies the claim.
\end{proof}

The next corollary follows from \Cref{rem:Dq_composite,thm:Dq_xi_measurable}.

\begin{corollary}\label{cor:Dq_random_variable}
$D_q(U_n,U)$ and $D_q(F_n',F)$ are random variables on $(\Omega',\A',P')$, and $D_q(F_n,F)$ is a random variable on $(\Omega,\A,P)$.
\end{corollary}

\begin{theorem}
For each $q\in\N_+$,
the probability measure on $\B(\R)$ induced by $D_q(U_n, U)$ and that by $D_q(F_n, F)$ are the same
if $F$ is continuous on $\R$.
\end{theorem}
\begin{proof}
For any $A\in\B(\R)$, we have 
\begin{align*}
P'\bigl(D_q(U_n,U)^{-1}(A)\bigr) 
&= P'\bigl(D_q(F_n',F)^{-1}(A)\bigr)\\
&= \bigl(P'\circ (X_1',\ldots,X_n')^{-1}\bigr)\Bigl(\wt{D}_{q,F}^{-1}(A)\Bigr)\\
&= \bigl(P\circ (X_1,\ldots,X_n)^{-1}\bigr)\Bigl(\wt{D}_{q,F}^{-1}(A)\Bigr)\\
&= P\bigl(D_q(F_n, F)^{-1}(A)\bigr)
\end{align*}
by \Cref{cor:P(Xd)=P(X),rem:Dq_composite,thm:Dq'=Dq,thm:Dq_xi_measurable,cor:Dq_random_variable}.
\end{proof}

This theorem implies \Cref{thm:D_q_distribution}.

\subsection{Proof of \texorpdfstring{\Cref{thm:D2_FnF_limit}}{Theorem \ref{thm:D2_FnF_limit}}}
\label{sec:proof_D2_FnF_limit}
\begin{definition}
For $F,G\in\F$, $x\in\R$, and $\bs{v} = (v_1, v_2) \in\R_\le^2$, define
\begin{align*}
\delta_{F,G}(x) &\coloneq F(x) - G(x), \\
d_{F,G}(\bs{v}) &\coloneq \lvert\delta_{F,G} (v_1)\rvert + \lvert\delta_{F,G} (v_2) - \delta_{F,G} (v_1)\rvert + \lvert\delta_{F,G} (v_2)\rvert.
\end{align*}
We also put $\overline{\delta}_{F,G} \coloneq \sup_{x\in\R} \delta_{F,G}(x)$ and $\underline{\delta}_{F,G} \coloneq \inf_{x\in\R} \delta_{F, G}(x)$.
Note that $\underline{\delta}_{F,G} \le 0 \le \overline{\delta}_{F,G}$ since $\lim_{x\to\pm\infty}\delta_{F,G}(x) = 0$.
\end{definition}

\begin{lemma}[cf. {\cite[Lemma 7.4]{komaba22}}]\label{lem:sup_d}
For any $F,G\in\F$, $\sup_{\bs{v}\in\R_\le^2} d_{F, G}(\bs{v}) = 2(\overline{\delta}_{F,G} - \underline{\delta}_{F,G})$.
\end{lemma}
\begin{proof}
Given $\bs{w} = (w_1, w_2) \in\R_\le^2$,
we have
\begin{align*}
	\max\{\delta_{F,G}(w_1), \delta_{F,G}(w_2)\} + \min\{\delta_{F,G}(w_1), \delta_{F,G}(w_2)\} &= \delta_{F,G}(w_1) + \delta_{F,G}(w_2),\\
	\max\{\delta_{F,G}(w_1), \delta_{F,G}(w_2)\} - \min\{\delta_{F,G}(w_1), \delta_{F,G}(w_2)\} &= \lvert\delta_{F,G}(w_1) - \delta_{F,G}(w_2)\rvert.
\end{align*}
If $\delta_{F,G}(w_1)>0$ and $\delta_{F,G}(w_2)>0$, then
\begin{align*}
	d_{F, G}(\bs{w}) 
	&= \delta_{F,G}(w_1) + \delta_{F,G}(w_2) + \lvert\delta (w_2) - \delta (w_1)\rvert \\
	&= 2\max\{\delta_{F,G}(w_1), \delta_{F,G}(w_2)\} \\
	&\le 2\overline{\delta}_{F,G} \\
	&\le 2(\overline{\delta}_{F,G} - \underline{\delta}_{F,G}).
\end{align*}
If $\delta_{F,G}(w_1)<0$ and $\delta_{F,G}(w_2)<0$, then
\begin{align*}
	d_{F, G}(\bs{w}) 
	&= -(\delta_{F,G}(w_1) + \delta_{F,G}(w_2)) + \lvert\delta (w_2) - \delta (w_1)\rvert \\
	&= -2\min\{\delta_{F,G}(w_1), \delta_{F,G}(w_2)\} \\
	&\le -2\underline{\delta}_{F,G} \\
	&\le 2(\overline{\delta}_{F,G} - \underline{\delta}_{F,G}).
\end{align*}
If $\delta_{F,G}(w_1)\delta_{F,G}(w_2)\le 0$, then $\lvert\delta_{F,G} (w_1)\rvert + \lvert\delta_{F,G} (w_2)\rvert = \lvert\delta_{F,G} (w_1) - \delta_{F,G} (w_2)\rvert$, hence
\begin{align*}
	d_{F, G}(\bs{w}) 
	&= 2\lvert\delta_{F,G}(w_1) - \delta_{F,G}(w_2)\rvert \\
	&\le 2(\overline{\delta}_{F,G} - \underline{\delta}_{F,G}).
\end{align*}
Taken together, $d_{F, G}(\bs{w}) \le 2(\overline{\delta}_{F,G} - \underline{\delta}_{F,G})$ holds in general.
Since $\bs{w}$ was arbitrary, we obtain $\sup_{\bs{v}\in\R_\le^2} d_{F, G}(\bs{v}) \le 2(\overline{\delta}_{F,G} - \underline{\delta}_{F,G})$.

On the other hand, for any $\epsilon>0$, there exist $w_1', w_2' \in\R$ such that $\overline{\delta}_{F,G} - \epsilon/4 < \delta_{F, G} (w_1')$ and $\underline{\delta}_{F,G} + \epsilon/4 > \delta_{F, G} (w_2')$.
Putting $w_1 = \min\{w_1', w_2'\}$, $w_2 = \max\{w_1', w_2'\}$, and $\bs{w} = (w_1, w_2)\in\R_\le^2$, we have 
\begin{align*}
2(\overline{\delta}_{F,G} - \underline{\delta}_{F,G}) 
&< 2(\delta_{F,G}(w_1') - \delta_{F,G}(w_2')) + \epsilon \\
&\le 2\lvert\delta_{F,G}(w_2) - \delta_{F,G}(w_1)\rvert + \epsilon \\
&\le d_{F,G}(\bs{w}) + \epsilon.
\end{align*}
Since $\epsilon$ was arbitrary, we obtain 
$2(\overline{\delta}_{F,G} - \underline{\delta}_{F,G}) \le \sup_{\bs{v}\in\R_\le^2} d_{F, G}(\bs{v})$.
\end{proof}

\begin{lemma}\label{lem:D2_delta}
For any $F, G \in\F$, $D_2(F, G) = \overline{\delta}_{F,G} - \underline{\delta}_{F,G}$.
\end{lemma}
\begin{proof}
We have
\begin{align*}
	D_2 (F, G) 
	&= \frac{1}{2}  \sup_{\bs{v}\in\R_{\le}^2} \sum_{i = 0}^2 \bigl\lvert F|_{v_i}^{v_{i + 1}} - G|_{v_i}^{v_{i + 1}}\bigr\rvert \\
	&= \frac{1}{2}  \sup_{\bs{v}\in\R_{\le}^2} d_{F, G}(\bs{v}) \\
	&= \overline{\delta}_{F,G} - \underline{\delta}_{F,G}
\end{align*}
by \Cref{rem:another_Dq_def,lem:sup_d}.
\end{proof}

\begin{definition}
{\normalfont (See \cite[pages 353 and 443]{dudley02} for reference.)}
Let $T$ be a set and $(\Omega, \A, P)$ a probability space.
A mapping $Y\colon T\times\Omega \to \R$ is called a \emph{stochastic process} 
if $Y_t \coloneq Y(t, {\,\cdot\,})\colon \Omega \to \R$ is measurable for each $t \in T$.
We say that $Y$ is Gaussian if $(Y_{t_1}, \ldots, Y_{t_m})\colon \Omega \to \R^m$ is Gaussian for any $t_1, \ldots, t_m \in T$.
\end{definition}

\begin{definition}
{\normalfont (See \cite[page 445]{dudley02} for reference.)}
Let $Y\colon T\times\Omega \to \R$ be a Gaussian stochastic process with $T = [0, 1]$.
If the following conditions hold:
\begin{itemize}
\item $\E[Y_t] = 0$ for any $t\in T$,
\item $\E[Y_s Y_t] = s(1-t)$ for any $s,t\in T$ with $s\le t$,
\item $Y$ is \emph{sample continuous}, i.e., $Y({\,\cdot\,}, \omega)\colon T\to \R$ is continuous for any $\omega \in \Omega$,
\end{itemize}
then $Y$ is called a \emph{Brownian bridge}.
\end{definition}

\begin{theorem}
{\normalfont \cite[Proposition 12.3.4]{dudley02}.}
For a Brownian bridge $Y$ and any $a>0$,
\begin{equation*}
P\biggl(\sup_{t\in[0, 1]} \lvert Y_t\rvert \ge a \biggr) = 2\sum_{i=1}^\infty (-1)^{i-1} \exp(-2i^2a^2).
\end{equation*}
\end{theorem}

\begin{theorem}\label{thm:supinf_brown_conv}
{\normalfont \cite[Proposition 12.3.6]{dudley02}.}
For a Brownian bridge $Y$ and any $a>0$,
\begin{equation*}
P\biggl(\sup_{t\in[0, 1]} Y_t - \inf_{t\in[0, 1]} Y_t \ge a \biggr) = 2\sum_{i=1}^\infty (4i^2 a^2 -1) \exp(-2i^2 a^2).
\end{equation*}
\end{theorem}

\begin{definition}
{\normalfont (See \cite[Section 12]{billing99} for reference.)}
Let $\D$ be the space of real functions on $[0, 1]$ that are right-continuous and have left-hand limits (such functions are called c\`adl\`ag functions).
Let $\Lambda$ be the set of strictly increasing, continuous mappings of $[0, 1]$ onto itself.
For $g, h\in\D$, define 
\begin{align*}
\lVert g\rVert &\coloneq \sup_{t\in[0, 1]} \lvert g(t)\rvert < \infty, \\
d(g, h) &\coloneq \inf_{\lambda \in \Lambda} \max \{\lVert\lambda - I\rVert, \lVert g - h\circ \lambda \rVert \} < \infty,
\end{align*}
where $I$ denotes the identity map on $[0, 1]$.
The function $d$ is a metric on $\D$, which defines the Skorohod topology.
\end{definition}

\begin{lemma}\label{lem:d_converge}
{\normalfont (See \cite[page 124]{billing99} for reference.)}
For an element $g$ and a sequence $\{g_n\}$ in $\D$, $\lim_{n\to\infty} g_n = g$ if and only if $\lim_{n\to\infty} \lVert\lambda_n - I\rVert = 0$ and $\lim_{n\to\infty} \lVert g - g_n\circ\lambda_n\rVert = 0$ for some sequence $\{\lambda_n\}$ in $\Lambda$.
\end{lemma}
\begin{proof}
If $\lim_{n\to\infty} g_n = g$, then $\lim_{n\to\infty}d(g, g_n) = \lim_{n\to\infty}\inf_{\lambda \in \Lambda} \max \{\lVert\lambda - I\rVert, \lVert g - g_n\circ \lambda \rVert \} = 0$.
Since there exists $\lambda_n\in\Lambda$ for each $n\in\N_+$ such that 
\begin{align*}
\inf_{\lambda \in \Lambda} \max \{\lVert\lambda - I\rVert, \lVert g - g_n\circ \lambda \rVert \} 
&\le \max \{\lVert\lambda_n - I\rVert, \lVert g - g_n\circ \lambda_n \rVert \} \\ 
&< \inf_{\lambda \in \Lambda} \max \{\lVert\lambda - I\rVert, \lVert g - g_n\circ \lambda \rVert \} + \frac{1}{n},
\end{align*}
$\lim_{n\to\infty} \lVert\lambda_n - I\rVert = 0$ and $\lim_{n\to\infty} \lVert g - g_n\circ\lambda_n\rVert = 0$ hold.

On the other hand, suppose 
$\lim_{n\to\infty} \lVert\lambda_n - I\rVert = 0$ and $\lim_{n\to\infty} \lVert g - g_n\circ\lambda_n\rVert = 0$ for some sequence $\{\lambda_n\}$ in $\Lambda$.
Then for any $\epsilon>0$, there exists $N\in\N$ such that $n>N$ implies $\lVert\lambda_n - I\rVert < \epsilon$ and $\lVert g - g_n\circ\lambda_n\rVert < \epsilon$, so that $\inf_{\lambda \in \Lambda} \max \{\lVert\lambda - I\rVert, \lVert g - g_n\circ \lambda \rVert \} < \epsilon$.
Hence $\lim_{n\to\infty} g_n = g$.
\end{proof}

\begin{lemma}\label{lem:sup_inf_D}
For $g, h \in \D$, the following inequalities hold:
\begin{align*}
\biggl\lvert \sup_{t\in[0, 1]} g(t) -  \sup_{t\in[0, 1]} h(t)  \biggr\rvert \le \lVert g - h\rVert,\qquad
\biggl\lvert \inf_{t\in[0, 1]} g(t) -  \inf_{t\in[0, 1]} h(t)  \biggr\rvert \le \lVert g - h\rVert.
\end{align*}
\end{lemma}

\begin{proof}
For any $\epsilon>0$, there exists $t'\in[0, 1]$ such that $\sup g(t) - \epsilon < g(t') \le \sup g(t)$.
Then $\sup g(t) - \sup h(t) < g(t') - \sup h(t) + \epsilon \le g(t') - h(t') + \epsilon \le \lVert g - h\rVert + \epsilon$.
Since $\epsilon$ was arbitrary, $\sup g(t) - \sup h(t) \le  \lVert g - h\rVert$ holds.
Similarly,  $\sup h(t) - \sup g(t) \le  \lVert g - h\rVert$ holds, and the first inequality is proved.
The second one can be proved in a similar way.
\end{proof}

\begin{theorem}\label{thm:D_continuous}
The function $\varphi\colon \D \to \R$ defined by $\varphi(g) \coloneq \sup_{t\in[0,1]} g(t) - \inf_{t\in[0,1]} g(t)$ is continuous.
\end{theorem}

\begin{proof}
For a given $g\in\D$,
let $\{g_n\}$ be a sequence in $\D$ such that $\lim_{n\to\infty} g_n = g$.
It suffices to prove that $\lim_{n\to\infty} \varphi(g_n) = \varphi(g)$.
By \Cref{lem:d_converge}, there exists a sequence $\{\lambda_n\}$ in $\Lambda$ such that $\lim_{n\to\infty} \lVert\lambda_n - I\rVert = 0$ and $\lim_{n\to\infty} \lVert g - g_n\circ\lambda_n\rVert = 0$.
By  \Cref{lem:sup_inf_D}, we have
\begin{align*}
\lvert\varphi(g_n) - \varphi(g)\rvert
&= \lvert\sup g_n(t) - \inf g_n(t) - \sup g(t) + \inf g(t)\rvert \\
&\le \lvert\sup g_n(t) - \sup g(t)\rvert +  \lvert\inf g_n(t)  - \inf g(t)\rvert \\
&= \lvert\sup g_n(\lambda_n(t)) - \sup g(t)\rvert +  \lvert\inf g_n(\lambda_n(t))  - \inf g(t)\rvert \\
&\le 2\lVert g_n\circ\lambda_n - g\rVert.
\end{align*}
Hence $\lim_{n\to\infty} \lvert\varphi(g_n) - \varphi(g)\rvert \le 2\lim_{n\to\infty} \lVert g_n\circ\lambda_n - g\rVert = 0$.
\end{proof}

\begin{definition}
{\normalfont (See \cite[pages 7 and 15--16]{billing99} for reference.)}
Suppose $P$ is a Borel probability measure and $\{P_n\}$ a sequence of Borel probability measures on a metric space $S$.
We say that $P_n$ \emph{converges weakly} to $P$ (denoted by $P_n \Rightarrow P$) if
$\lim_{n\to\infty}P_n g = P g$
for all bounded continuous functions $g\colon S\to\R$,
where $Pg \coloneq \int_S g\: \d P$.
A set $A\subset S$ whose boundary $\partial A$ satisfies $P(\partial A) = 0$ is called a $P$-\emph{continuity  set}.
\end{definition}

\begin{theorem}\label{thm:Pn_weak_convergence}
{\normalfont \cite[Theorem 2.1]{billing99}.}
Suppose $P$ is a Borel probability measure and $\{P_n\}$ a sequence of Borel probability measures on a metric space $S$.
Then these five conditions are equivalent:
\begin{enumerate}[label=\textnormal{(\roman*)}]
\item $P_n \Rightarrow P$.
\item $\lim_{n\to\infty}P_n g = Pg$ for all bounded, uniformly continuous functions $g\colon S\to\R$.
\item $\limsup_{n\to\infty} P_n(K) \le P(K)$ for any closed set $K\subset S$.
\item $\liminf_{n\to\infty} P_n(V) \ge P(V)$ for any open set $V\subset S$.
\item $\lim_{n\to\infty}P_n(A) = P(A)$ for any $P$-continuity set $A\subset S$. \label{thm:Pn_weak_convergence_v}
\end{enumerate}
\end{theorem}

\begin{definition}
{\normalfont (See \cite[pages 24--26]{billing99} for reference.)}
We call a map from a probability space $(\Omega, \A, P)$ to a metric space $S$
a \emph{random element} if it is Borel measurable.
(As is customary, we call it a random variable if, in addition, $S=\R$.)
Suppose $Z$ is a random element and $\{Z_n\}$ a sequence of random elements from $(\Omega, \A, P)$ to $S$.
The \emph{law} of $Z$ is the Borel probability measure $L_Z \coloneq P\circ Z^{-1}$ on $S$.
We say that $Z_n$ \emph{converges in distribution} to $Z$ (denoted by $Z_n \Rightarrow Z$) if
$L_{Z_n} \Rightarrow L_Z$.
A set $A\subset S$ with $P(Z^{-1}(\partial A)) = 0$ is called a \emph{$Z$-continuity set}.
\end{definition}

\begin{theorem}\label{thm:Zn_convergence_in_distribution}
{\normalfont \cite[page 26]{billing99}.}
Suppose $Z$ is a random element and $\{Z_n\}$ a sequence of random elements from a probability space $(\Omega, \A, P)$ to a metric space $S$.
Then these five conditions are equivalent:
\begin{enumerate}[label=\textnormal{(\roman*)}]
\item $Z_n \Rightarrow Z$.
\item $\lim_{n\to\infty}\E[g(Z_n)] = \E[g(Z)]$ for all bounded, uniformly continuous functions $g\colon S\to\R$.
\item $\limsup_{n\to\infty} P(Z_n\in K) \le P(Z\in K)$ for any closed set $K\subset S$.
\item $\liminf_{n\to\infty} P(Z_n\in V) \ge P(Z\in V)$ for any open set $V\subset S$.
\item $\lim_{n\to\infty}P(Z_n\in A) = P(Z\in A)$ for any $Z$-continuity set $A\subset S$.
\end{enumerate}
\end{theorem}

\begin{theorem}\label{thm:ifUniformDist_converge}
{\normalfont \cite[page 20]{billing99}.}
Suppose $P$ is a Borel probability measure and $\{P_n\}$ a sequence of Borel probability measures on a metric space $S$.
Let $S'$ be another metric space and $h\colon S\to S'$ a continuous map.
If $P_n \Rightarrow P$, then $P_n\circ h^{-1} \Rightarrow P\circ h^{-1}$.
\end{theorem}

\begin{remark}[See {\cite[page 135]{billing99}} for reference]
  Note that $Y \colon (\Omega, \mathcal A, P) \to D[0, 1]$
  is a random element if and only if
  $Y_{\bullet} \colon [0, 1] \times \Omega \to \R$ is a stochastic process
  (i.e., $Y_t\colon \Omega \to \R$ is a random variable for all $t \in [0, 1]$).
\end{remark}

\begin{theorem}\label{thm:D2_FnF_limit_proof}
Suppose $X_1,\ldots,X_n$ are i.i.d.\ random variables on a probability space $(\Omega, \A, P)$ with a continuous distribution function $F\in\F$, and $F_n(t) = \sum_{i=1}^n\1_{(-\infty,t]}(X_i)/n$ for $t\in\R$.
Then
\begin{equation}\label{eq:ifUniformDist_lim}
\lim_{n\to\infty} P\biggl(D_2(F_n, F) \ge \frac{a}{\sqrt{n}}\biggr) = 2\sum_{i=1}^\infty (4i^2 a^2 -1) \exp( -2i^2 a^2)
\end{equation}
for any $a>0$.
\end{theorem}
\begin{proof}
Let $U_n$ be defined as in \Cref{sec:proof.D_q_distribution} on $(\Omega, \A, P)$, $Y_t^n = \sqrt{n}(U_n(t) - U(t))$ for $t\in[0, 1]$, and $Y\colon [0, 1]\times\Omega \to \R$ a Brownian bridge.
Then $Y^n$ and $Y$ are random elements on $(\Omega, \A, P)$ to $D[0, 1]$, and 
$Y^n \Rightarrow Y$ (which means $L_{Y^n}\Rightarrow L_{Y}$) by \cite[Theorem 14.3]{billing99}.
Since 
\begin{equation*}
L_{Y^n}\circ\varphi^{-1}([a, \infty)) = P\biggl(\sup_{t\in[0, 1]}Y_t^n - \inf_{t\in[0, 1]}Y_t^n \ge a\biggr),
\end{equation*}
\begin{align*}
L_{Y}\circ\varphi^{-1}([a, \infty)) &= P\biggl(\sup_{t\in[0, 1]}Y_t - \inf_{t\in[0, 1]}Y_t \ge a\biggr)\\
&= 2\sum_{i=1}^\infty (4i^2 a^2 -1) \exp(-2i^2 a^2)
\end{align*}
for any $a>0$
by \cref{thm:supinf_brown_conv}, and $\lim_{n\to\infty}(L_{Y^n}\circ\varphi^{-1})([a, \infty))= (L_{Y}\circ\varphi^{-1})([a, \infty))$ by \cref{thm:D_continuous,thm:ifUniformDist_converge,thm:Pn_weak_convergence}, we obtain 
\begin{equation}\label{eq:lim_supinf_a}
\lim_{n\to\infty} P\biggl(\sup_{t\in[0, 1]}Y_t^n - \inf_{t\in[0, 1]}Y_t^n \ge a\biggr) = 2\sum_{i=1}^\infty (4i^2 a^2 -1) \exp(-2i^2 a^2) \qquad (a>0).
\end{equation}
Note that $[a, \infty)$ is an $L_Y \circ \varphi^{-1}$-continuity set.
It holds almost surely that $\sup_{t\in[0, 1]}Y_t^n = \sup_{t\in (0, 1)}Y_t^n$, $\inf_{t\in[0, 1]}Y_t^n = \inf_{t\in (0, 1)}Y_t^n$.
The continuity of $F$ implies that $(0, 1)\subset F(\R) \subset [0, 1]$.
Therefore, we have 
\begin{align*}
&P\biggl(\sup_{t\in[0, 1]}Y_t^n - \inf_{t\in[0, 1]}Y_t^n \ge a\biggr)\\
&= P\biggl(\sup_{x\in\R}Y_{F(x)}^n - \inf_{x\in\R}Y_{F(x)}^n \ge a\biggr)\\
&= P\biggl(\sup_{x\in\R} (U_n(F(x)) - U(F(x)))- \inf_{x\in\R} (U_n(F(x)) - U(F(x)))\ge \frac{a}{\sqrt{n}}\biggr)\\
&= P\biggl(\sup_{x\in\R} (F_n(x) - F(x))- \inf_{x\in\R} (F_n(x) - F(x))\ge \frac{a}{\sqrt{n}}\biggr)\\
&= P\biggl(D_2(F_n, F)\ge \frac{a}{\sqrt{n}}\biggr)
\end{align*}
by \Cref{cor:P(Xd)=P(X),cor:Fn=Un_F,lem:D2_delta}.
With \Cref{eq:lim_supinf_a}, we obtain \Cref{eq:ifUniformDist_lim}.
\end{proof}

This theorem is equivalent to \Cref{thm:D2_FnF_limit}.

\section*{Acknowledgement}
We thank Akitomo Amakawa (University of Yamanashi) for the validation of our numerical results.

\section*{Funding}
This study was partially supported by JSPS KAKENHI Grant Numbers JP21K15762, JP20K03509, and JP24K06686.

\section*{Supplementary Material}
\label{supplementary-material}
The source code of the experiment in \Cref{sec:experiments}.
The following packages are required:
\begin{itemize}[nosep]
  \item \texttt{numpy==1.26.4}
  \item \texttt{pandas==2.2.2}
  \item \texttt{scipy==1.13.0}
  \item \texttt{matplotlib==3.8.4}
\end{itemize}
This work \textcopyright~2024 by Atsushi Komaba is licensed under Creative Commons Attribution 4.0 International. To view a copy of this license, visit \url{https://creativecommons.org/licenses/by/4.0/}
\lstset{
  language=Python,
  basicstyle={\footnotesize\ttfamily},
  keywordstyle={\color{blue}},
}
\begin{lstlisting}
# Copyright (c) 2024 Atsushi Komaba

import math
import numpy as np
import pandas as pd
from scipy import stats
from matplotlib import pyplot as plt

# The number of d2 values computed to obtain its empirical distribution
num_d2s_under_h0 = 100000
# Repetition time for the experiments
num_repeat = 100000
# Sample sizes used in the experiments
sample_sizes = [2**i for i in range(3, 13)]


# Distributions used in the experiments
class Mixture:
  def rvs(self, size, random_state):
    ret = 1.0 * random_state.integers(2, size=size)
    i = ret == 0
    ret[i] = stats.norm.rvs(
      -4 / 5,
      3 / 5,
      ret[i].size,
      random_state=random_state,
    )
    ret[~i] = stats.norm.rvs(
      4 / 5,
      3 / 5,
      ret[~i].size,
      random_state=random_state,
    )
    return ret

  def cdf(self, xs):
    return (
      stats.norm.cdf(xs, -4 / 5, 3 / 5) + stats.norm.cdf(xs, 4 / 5, 3 / 5)
    ) / 2

  def pdf(self, xs):
    return (
      stats.norm.pdf(xs, -4 / 5, 3 / 5) + stats.norm.pdf(xs, 4 / 5, 3 / 5)
    ) / 2


dists = {
  "Normal(0, 1)": stats.norm(),
  "Normal(0.2, 1)": stats.norm(0.2, 1),
  "Normal(0, 1.1)": stats.norm(0, 1.1),
  "Trapezoidal": stats.trapezoid(
    (2 - math.sqrt(2)) / 4,
    (2 + math.sqrt(2)) / 4,
    -2,
    4,
  ),
  "Mixture": Mixture(),
}
dist_pairs = [
  ["Normal(0.2, 1)", "Normal(0, 1)"],
  ["Normal(0, 1.1)", "Normal(0, 1)"],
  ["Trapezoidal", "Normal(0, 1)"],
  ["Mixture", "Normal(0, 1)"],
  ["Trapezoidal", "Mixture"],
  ["Mixture", "Trapezoidal"],
]
# Significance level
significance_level = 0.05

# Random generator with fixed seed
rng = np.random.default_rng(seed=0)


# Computes the D2 statistic between `cdf`
# and the empirical distribution function based on `xs`.
def d2(xs, cdf, axis=0):
  ys = np.repeat(cdf(np.sort(xs, axis)), 2, axis)
  zs = np.repeat(np.linspace(0, 1, xs.shape[axis] + 1), 2)
  ds = ys - zs[(slice(1, -1),) + (None,) * (xs.ndim - axis - 1)]
  return np.max(ds, axis) - np.min(ds, axis)


# Precomputation of D2 values under the null hypothesis H0.
# Compares samples from the standard uniform distribution U with U itself
uniform = stats.uniform()
d2s_under_h0 = {}
print(f"# Precomputation of D2 under H0 ({num_d2s_under_h0:,} each)")
for sample_size in sample_sizes:
  print(
    f"Sample size: {sample_size:,} ... ",
    end="",
    flush=True,
  )
  xs = uniform.rvs((sample_size, num_d2s_under_h0), random_state=rng)
  d2s_under_h0[sample_size] = np.sort(d2(xs, uniform.cdf))
  print("Done")


# Draws the graphs of the empirical distribution functions obtained above.
def draw_d2_under_h0():
  sample_sizes = [2**i for i in range(9, 2, -2)]
  num_sample_sizes = len(sample_sizes)
  fig = plt.figure(figsize=(8, 3))
  ax = fig.subplots()
  ax.set_xlim(0.5, 2.5)
  # Computes the theoretical asymptotic distribution function.
  # Relative errors <= 1e-10.
  a = np.linspace(0.5, 2.5, 1000)
  i = np.arange(1, 10)
  sq = np.outer(a, i) ** 2
  asymp_cdf = 1 - 2 * np.sum((4 * sq - 1) * np.exp(-2 * sq), axis=1)
  asymp_cdf[a < 0.4] = 0
  ax.plot(
    a,
    asymp_cdf,
    color=str(1 - 1 / (num_sample_sizes + 1)),
    label=r"$n \rightarrow \infty$ (asymptotic)",
  )
  # Expirical distribution function of D2 under H0.
  for i, sample_size in enumerate(sample_sizes):
    ax.ecdf(
      d2s_under_h0[sample_size] * math.sqrt(sample_size),
      label=f"$n = {sample_size}$ (empirical)",
      color=str((num_sample_sizes - i - 1) / (num_sample_sizes + 1)),
      linestyle=(0, (num_sample_sizes - i, 1)),
    )
  handles, labels = ax.get_legend_handles_labels()
  ax.legend(handles[::-1], labels[::-1])
  ax.set_xlabel(r"$\sqrt{n}D_2(U_n, U)$")
  fig.tight_layout()
  fig.savefig("d2_under_h0.pdf")


draw_d2_under_h0()

# PDFs
fig = plt.figure(figsize=(6.4, 3.2))


def dists_pdf(ax, dists_name):
  for name, linestyle in zip(dists_name, ["-", "--", ":"]):
    xs = np.arange(-4, 4, 0.01)
    ax.plot(
      xs,
      dists[name].pdf(xs),
      label=name,
      color="black",
      linestyle=linestyle,
    )
  ax.legend()


dists_pdf(
  fig.add_subplot(211),
  ["Normal(0, 1)", "Normal(0.2, 1)", "Normal(0, 1.1)"],
)
dists_pdf(
  fig.add_subplot(212),
  ["Normal(0, 1)", "Trapezoidal", "Mixture"],
)
fig.tight_layout()
fig.savefig("dists_pdf.pdf")

print(f"# The experiment ({num_repeat:,} trial each)")
fig = plt.figure(figsize=(2 * 6.4, 2 * 4.8))
for idx, [dist_name, ref_name] in enumerate(dist_pairs):
  print(f"Sampling distribution: {dist_name}, reference distribution: {ref_name}")
  dist = dists[dist_name]
  ref = dists[ref_name].cdf
  ax = fig.add_subplot(3, 2, idx + 1, xscale="log")
  statistical_powers = pd.DataFrame()
  for sample_size in sample_sizes:
    print(f"  Sample size: {sample_size:,}")
    xs = dist.rvs((sample_size, num_repeat), random_state=rng)

    print(
      "    Kolmogorov-Smirnov test: ",
      end="",
      flush=True,
    )
    pvalue = stats.ks_1samp(xs, ref).pvalue
    num_rejected = np.count_nonzero(pvalue < significance_level)
    statistical_powers.loc[sample_size, "Kolmogorov-Smirnov"] = (
      num_rejected / num_repeat
    )
    print(f"rejected {num_rejected:,} out of {num_repeat:,}")

    print(
      "    Cram\xe9r-von Mises test: ",
      end="",
      flush=True,
    )
    pvalue = stats.cramervonmises(xs, ref).pvalue
    num_rejected = np.count_nonzero(pvalue < significance_level)
    statistical_powers.loc[sample_size, "Cram\xe9r-von Mises"] = (
      num_rejected / num_repeat
    )
    print(f"rejected {num_rejected:,} out of {num_repeat:,}")

    print("    OVL-2: ", end="", flush=True)
    pvalue = (
      1
      - np.searchsorted(d2s_under_h0[sample_size], d2(xs, ref)) / num_d2s_under_h0
    )
    num_rejected = np.count_nonzero(pvalue < significance_level)
    statistical_powers.loc[sample_size, "OVL-2"] = num_rejected / num_repeat
    print(f"rejected {num_rejected:,} out of {num_repeat:,}")
  ax.set_title(
    f"Sampling distribution: {dist_name}, reference distribution: {ref_name}"
  )
  ax.set_ylim(-0.05, 1.05)
  statistical_powers.plot(ax=ax, style=[".-k", "x--k", "+:k"])
  ax.legend()
fig.tight_layout()
fig.savefig("result.pdf")
\end{lstlisting}
\label{supplementary-material-end}

\bibliographystyle{abbrv}
\bibliography{bibliography}

\end{document}